\documentclass[11pt,a4paper]{amsart}

\usepackage[latin1]{inputenc}

\usepackage{amssymb,amsmath,amscd,amsfonts,amsthm,mathrsfs}
\usepackage{times} 

\usepackage[all]{xy}
\usepackage{graphicx}

\usepackage{color}

\usepackage[colorlinks=true, linkcolor=black, urlcolor=blue, citecolor=black, breaklinks, pagebackref]{hyperref}

\newcommand{\Cc}{\mathcal{C}}

\newcommand{\ess}{{\rm{ess}}}

\newcommand{\N}{\mathbb{N}}
\newcommand{\R}{\mathbb{R}}
\newcommand{\C}{\mathbb{C}}

\newcommand{\D}{\mathbb{D}}
\newcommand{\T}{\mathbb{T}}
\newcommand{\ds}{\displaystyle}
\newcommand{\si}{\sigma}
\newcommand{\al}{\alpha}
\newcommand{\la}{\lambda}
\newcommand{\ap}{\approx}

\newcommand{\Sg}{\textbf{S}}
\newcommand{\Id}{{\rm{Id}}}

\newcommand{\bsl}{\backslash}

\newtheorem{theorem}{Theorem}[section]
\newtheorem{proposition}[theorem]{Proposition}
\newtheorem{lemma}[theorem]{Lemma}

\newtheorem{remark}[theorem]{Remark}

\numberwithin{equation}{section}

\title[Lieb-Thirring inequalities for fractional Schr\"odinger operator]{Note on Lieb-Thirring type inequalities for a complex perturbation of fractional Laplacian}

\begin{document}
\author{Cl\'{e}ment Dubuisson}
\address{Institut de Math\'{e}matiques de Bordeaux
Universit\'{e} de Bordeaux
351, cours de la Lib\'eration
F-33405 Talence cedex}
\email{clement.dubuisson@math.u-bordeaux.fr}

\begin{abstract}
For $s>0$, let $H_0=(-\Delta)^s$ be the fractional Laplacian. In this paper, we obtain
Lieb-Thirring type inequalities for the fractional Schr\"odinger operator defined as
\[
H=H_0+V,
\]
where $V\in L^p(\mathbb{R}^d), p\ge 1, d\ge 1,$ is a complex-valued potential.
Our methods are based on results of articles by Borichev-Golinskii-Kupin \cite{BoGoKu} and Hansmann \cite{Ha1}.
\end{abstract}

\maketitle

\section{Introduction.}

The article by Abramov-Aslanyan-Davies \cite{AbAsDa} gave rise to many papers devoted to the study of eigenvalues  for complex perturbations of various self-adjoint  operators.
Recently, the fractional Laplacian $(-\Delta)^s, s>0$, received an increasing interest  due to its numerous applications in applied mathematics and physics (see \cite{DiNPaVa} for references).

For $s>0$, the fractional Laplacian $(-\Delta)^s$ is defined with the help of the functional calculus applied to the nonnegative self-adjoint operator \ $-\Delta$.
That is, $(-\Delta)^s$ is essentially self-adjoint on $\Cc^\infty_c(\R^d;\C)$,
and the domain of its closure is the fractional Sobolev space
\[
W^{s,2}(\R^d,\C):=\{f \in L^2(\R^d), \ds\int_{\R^d}(1+|\zeta|^{2s}) |\widehat{f}(\zeta)|^2 d\zeta < +\infty\},
\]
where $\widehat{f}$ is the Fourier transform of $f$
(see \cite[Section 3.1]{DiNPaVa}).
By the spectral mapping Theorem, the spectrum of $(-\Delta)^s$ is $\R^+=[0;+\infty[$.

We consider the fractional Schr\"odinger operator
\begin{equation}\label{def:H}
H=(-\Delta)^s + V,
\end{equation}
where $V$ is the operator of multiplication by the complex-valued function $V$
and we note $H_0:=(-\Delta)^s$.
In particular, the perturbed operator $H$ is not supposed to be self-adjoint.
We assume that $V$ is a relatively compact perturbation of $H_0$,
\textit{i.e.} $\mathrm{dom}(H_0) \subset \mathrm{dom}(V)$ and $V(\la-H_0)^{-1}$ is compact for $\la \in \C\backslash \si(H_0)$.
The spectrum, the essential spectrum, and the discrete spectrum of $H$ will be denoted by
$\si(H), \si_{\ess}(H),$ and $\si_d(H)$, respectively.
Here, the discrete spectrum is the set of all eigenvalues which are discrete points of the spectrum whose corresponding eigenspaces (or rootspaces) are finite dimensional.
Throughout this work, eigenvalues are counted according to their algebraic multiplicity.
The essential spectrum of $H$ is defined by 
$\si_{\ess}(H)=\{\la \in \C, \la-H \mbox{ is not Fredholm}\}$,
where a closed operator is a Fredholm operator if it has a closed range and both its kernel and cokernel are finite dimensional.
In our situation, $\si_{\ess}(H):=\si(H)\bsl \si_d(H)$.
For more details on these definitions, see \cite[Chapter IX]{EdEv}.
By Weyl's Theorem on essential spectrum (see \cite[Theorem XIII.14]{ReSi4}), we have 
\begin{equation*}
\si_{\ess}(H)=\si_{\ess}(H_0)=\si(H_0)=\R^+,
\end{equation*}
and the only possible accumulation points of $\si_d(H)$ lie on $\si_{\ess}(H)$.

Our interest in the present topic was motivated by the article of Frank-Lieb-Seiringer
\cite{FrLiSe} on Hardy-Lieb-Thirring inequalities for fractional Schr\"odinger operator with real-valued potential $V$.
As an application, the authors give a proof of the stability of relativistic matter.
In particular, for $0<s< \min\{1;\frac{d}{2}\}$, $\gamma >0$, and $V_- \in L^{\gamma+d/2s}(\R^d)$, formula (5.11) in \cite{FrLiSe} says that
\begin{equation}\label{i:FrLiSe1}
\ds\sum_{\la\in \si_d(H)} |\la|^{\gamma}\leq C_{d,s,\gamma}
 \cdot \big\|V_-\big\|^{\gamma+d/2s}_{L^{\gamma+d/2s}},
\end{equation}
where $V_-=\max\{0;-V\}$ and $C_{d,s,\gamma}$ is defined at \cite[formula (5.11)]{FrLiSe}.

In this paper, we obtain Lieb-Thirring type inequalities for the fractional Schr\"o\-din\-ger operator $H$ with complex-valued $V$.
These inequalities give information on the rate of convergence of points from the discrete spectrum $\si_d(H)$  to the essential spectrum of $H$.
The pertaining references on the subject are \cite{FrLaLiSe}, \cite{HuSi}, and \cite{DeHaKa1}.

We will assume a little more than $V$ being a relatively compact perturbation of $H_0$. Actually, we will suppose that $V$ is relatively Schatten-von Neumann perturbation of the fractional Laplacian.
Namely, let  $\Sg_p, p\geq 1$, be the Schatten-von Neumann class of compact operators, see Section \ref{s2.3} for further references on the subject.
Saying that the potential $V$ defined on $\R^d$ is a relatively Schatten-von Neumann perturbation of $H_0$ means that
$\mathrm{dom}(H_0) \subset \mathrm{dom}(V)$ and
\begin{equation}\label{hyp}
 V(\la-H_0)^{-1} \in \Sg_p,
\end{equation}
for one (and hence for all) $\la \in \C\backslash\si(H_0)$.
Hypothesis \eqref{hyp} is fulfilled provided $V\in L^p(\R^d)$ and $p > \max\{1;\frac{d}{2s}\}$ (see Proposition \ref{det-BS}).


We denote by $d(z, \Omega):=\ds\inf_{w\in \Omega}|z-w|$ the distance between $z \in \C$ and $\Omega \subset \C$. As usual,  $x_+=\max\{0;x\}$.
The main results of the present article are the following theorems.
The constants $\omega$ and $C_{\omega}$ are defined in \eqref{maj:opp}.
\begin{theorem}\label{T1}
Let $H$ be the fractional Schr\"{o}dinger operator defined by \eqref{def:H} with $0<s\leq \frac{d}{2}$
and $V \in L^p(\R^d)$ with $p>\frac{d}{2s}$.
Then, for $\tau>0$ small enough, we have
\begin{equation}\label{t11}
\ds\sum_{\la \in \si_d(H)} \dfrac{d(\la,\si(H_0))^{p+\tau}}
  {|\la|^{\al}(1+|\la|)^{\beta}}
\leq K \cdot \dfrac{C_{\omega}^p\, \omega^{\beta-\tau}}{\tau} \cdot \|V\|_{L^p}^p,
\end{equation}
where the powers are
\begin{enumerate}
\item $\al=\min\{\frac{p+\tau}{2}; \frac{d}{2s}\}$,
\item $\beta =2\tau+\frac{1}{2}(\frac{d}{s}-p-\tau)_+$.
\end{enumerate}
The constant $K$ depends on $d, p, s$, and $\tau$.
\end{theorem}

\begin{theorem}\label{T1b}
Let $H$ be the fractional Schr\"{o}dinger operator defined by \eqref{def:H} with $s>\frac{d}{2}$
and $V \in L^p(\R^d)$ with $p > 1$.
Then, for $\tau >0 $ small enough, we have
\begin{equation}\label{t12}
\ds\sum_{\la \in \si_d(H)}\dfrac{d(\la,\si(H_0))^{p+1-\frac{d}{2s}+\tau}}
{|\la|^{\al}(1+|\la|)^{\beta}}
\leq K' \cdot \dfrac{C_{\omega}^p\, \omega^{\beta-\tau}}{\tau} \cdot \|V\|_{L^p}^p,
\end{equation}
where the powers are
\begin{enumerate}
\item $\al=\frac{1}{2}+\frac{1}{2}\min\{p-\frac{d}{2s}+\tau;1\}$,
\item $\beta = 2\tau +\frac{1}{2}(\frac{d}{2s}-p+1-\tau)_+$.
\end{enumerate} 
The constant $K'$ depends on $d, p, s$, and $\tau$.
\end{theorem}

The above theorems essentially rely on complex analysis methods presented in \cite{BoGoKu}, while Theorem \ref{T2} is based on results of \cite{Ha1}, obtained with the help of tools of functional analysis and operator theory.

\begin{theorem}\label{T2}
Let $H$ be the fractional Schr\"{o}dinger operator
defined in \eqref{def:H} with $s>0$ and $V\in L^p(\R^d)$,
with $p>\max\{1;\frac{d}{2s}\}$.
Then, for $\tau>0$, the following inequality holds
\begin{equation}\label{t2}
 \ds\sum_{\la \in \si_d(H)} \dfrac{d(\la,\si(H_0))^{p}}
{(1+|\la|)^{\frac{d}{2s}+\tau}}
\leq K'' \cdot \dfrac{C_{\omega}^p\, \omega^{\frac{d}{2s}}}{\tau}  \cdot \|V\|^p_{L^p},
\end{equation}
with $K''$ depending on $d, p, s$, and $\tau$.
\end{theorem}


Let us assume that $V$ is real-valued and $V \in L^p(\R^d)$ for $p > \max\{1; \frac{d}{2s}\}$.
Then $H=(-\Delta)^s+V$ is a self-adjoint operator
and $\si_d(H)$ lies on the negative real half-axis.

In \cite{FrLiSe}, the values of the parameter $s$ are restricted to the range
$0<s < \min\{1, \frac{d}{2}\}$ due to the presence of the magnetic potential (see \cite[Section 2.1]{FrLiSe}).
Setting $\gamma=p-\frac{d}{2s}>0$, \eqref{i:FrLiSe1} becomes
\begin{equation}\label{i:FrLiSe}
\sum_{\la\in \si_d(H)} |\la|^{p-\frac{d}{2s}}\leq C_{p,d,s} \, \|V_-\|^p_{L^p},
\end{equation}
where, as always, $V_-=\min\{V, 0\}$.

Theorems \ref{T1} and \ref{T1b} give Lieb-Thirring type inequalities for all positive values of $s$.
In particular, for $0< s \leq \frac{d}{2}$, \eqref{t11} becomes
\begin{equation}\label{LT-sa}
\ds\sum_{\la \in \si_d(H)} |\la|^{\max\{\frac{p+\tau}{2}; p-\frac{d}{2s}+\tau\}} \leq C_{d,p,s, \omega} \| V\|^p_{L^p}.
\end{equation}
Rather expectedly,  we see that \eqref{LT-sa} is slightly weaker than \eqref{i:FrLiSe}, but our results apply to a considerably larger class of potentials.


We continue with few words on the notation.
The generic constants will be denoted  by $C$, that is, they will be allowed to change from one relation to another.
For two positive functions $f, g$ defined on a
domain $\Omega$ of the complex plane $\C$, we write $f(\la) \ap
g(\la)$ if there are constants $C_1, C_2 >0 $ such that $C_1
f(\la)\leq g(\la)\leq C_2 f(\la)$, for all $\la\in \Omega$.
We write $f(\la) \lesssim g(\la)$ ( $f(\la) \gtrsim g(\la)$) if there is a positive constant $C$ such that
$f(\la)\leq C\cdot g(\la)$ ($f(\la)\geq C\cdot g(\la)$, respectively) for $\la \in \Omega$.
The choice of the domain $\Omega$ will be clear from the context.


To compare Theorem \ref{T1} and Theorem \ref{T2},
it is convenient to consider a sequence $(\la_n) \subset \si_d(H)$ which converges to $\la \in \si_{\ess}(H)$.
Recall that $\si_{\ess}(H)=\si(H_0)$.
We give details on the comparison between \eqref{t11} and \eqref{t2},
the comparison between \eqref{t12} and \eqref{t2} being similar.
Without the loss of generality, we assume $d(\la_n,\si(H_0))\leq 1$.

In the case  $\la\in ]0; +\infty[$,  \eqref{t2} is better than \eqref{t11}.

In the case $\la=\infty$,
the term in \eqref{t11} becomes
$\dfrac{d(\la_n,\si(H_0))^{p+\tau}}{|\la_n|^{\frac{d}{2s}+2\tau}}$, for $n$ large enough,
and \eqref{t2} is again better than \eqref{t11}.

When $\la=0$, the situation is slightly more complicated.
Choose $\tau>0$ small enough to guarantee $\alpha/\tau\geq 1$. Then, \eqref{t11} is better than \eqref{t2} provided $\mathrm{Re}(\la_n)\leq 0$ \ or \  $|\mathrm{Im}(\la_n)|^\tau \gtrsim |\mathrm{Re}(\la_n)|^\alpha$ \  for \ $\mathrm{Re}(\la_n)> 0$. Inequality \eqref{t2} is better than \eqref{t11} in the opposite case, {\it i.e.,} when $\mathrm{Re}(\la_n)> 0$ \ and \ $|\mathrm{Im}(\la_n)|^\tau \lesssim |\mathrm{Re}(\la_n)|^\alpha$.

To sum up,  we see that neither Theorem \ref{T1} nor Theorem \ref{T2} take the advantage over each other.


As concluding remark, we would like to mention that it is possible to consider complex matrix-valued potential $V$ as in \cite{Du}, devoted to the study of non-self-adjoint Dirac operators.  Unlike the latter paper, the study of matrix-valued potentials is neither natural nor complicated in the present framework.
The only difference with the scalar-valued case will be the presence of the constant $n^{p/2}$ in \eqref{BR} and \eqref{BR1},
$n$ being the size of the square matrix giving the matrix-valued potential.

At last, we say few words on the structure of the paper.
We recall some known results and give references in Section \ref{s-prelim}.
The key point of the proofs is the bound on the resolvent of $H_0$, and it is proved in Section \ref{s-BS}.
In Sections \ref{s-proof1} and \ref{s-proof1b} we prove Theorems \ref{T1} and \ref{T1b}, respectively.
In Section \ref{s-proofH} we deal with Theorem \ref{T2}.

\medskip

\noindent\textbf{Acknowledgment}: 
I thank Stanislas Kupin for his helpful comments on the subject.


\section{Preliminaries.}\label{s-prelim}

\subsection{Theorem of Borichev-Golinskii-Kupin.}
The following theorem, proved in \cite[Theorem 0.2]{BoGoKu}, gives a bound on the distribution of zeros of a holomorphic function on the unit disc $\D:=\{|z|<1\}$ in terms of its growth towards the boundary $\T:=\{|z|=1\}$.

\begin{theorem}\label{BGK}
Let $h$ be a holomorphic function on $\D$ with $h(0)=1$.
Assume that $h$ satisfies a bound of the form
\begin{equation}\label{hyp:BGK}
 |h(z)|\leq\exp\left(
\dfrac{K}{(1-|z|)^{\al}} \ds\prod_{j=1}^N\dfrac{1}{|z-\zeta_j|^{\beta_j}}
\right),
\end{equation}
where $|\zeta_j|=1$ and $\al, \beta_j\geq 0, \ j=1,\dots, N$.

Then, for any $0<\tau<1$, the zeros of $h$ satisfy the inequality
\[ \ds\sum_{h(z)=0}(1-|z|)^{\al+1+\tau}
\ds\prod_{j=1}^N|z-\zeta_j|^{(\beta_j-1+\tau)_+}\leq C\cdot K,
\]
where $C$ depends on $\al,\beta_j,\zeta_j$ and $\tau$.
\end{theorem}
Above, $x_+=\max\{x,0\}$.


\subsection{Conformal maps.}

Let $\varphi_a$ be a conformal map sending $\D$ to the resolvent set of the operator $H_0$, $\rho(H_0)=\C\bsl\R^+$. For $a>0$, it is given by the relation
\begin{equation}\label{phi}
\varphi_a: z \mapsto \la:=-a\left(\dfrac{z+1}{z-1} \right)^2,
\end{equation}
and the inverse map going from $\C\bsl \R^+$ to $\D$ is
\begin{equation*}
\varphi_a^{-1}:\la \mapsto z:=\dfrac{\sqrt{\la}-\mathrm{i}\sqrt{a}}{\sqrt{\la}+\mathrm{i}\sqrt{a}}.
\end{equation*}

Later in the paper, we will have to compare the distance from $\la=\varphi_a(z)$
to the boundary of $\rho(H_0), \  \partial \rho(H_0)=\R^+$,
and the distance from $z$ to $\partial \D=\T$.
The results of this kind are called distortion theorems.

\begin{proposition}[Distortion between $\C\bsl \R^+$ and $\D$]\label{CM:2}
Let, as above,  $\la=\varphi_a(z)$. We have
\begin{equation}\label{cm21}
a \cdot d(z,\T) \cdot \dfrac{|z+1|}{|z-1|^3} \leq d(\la,\R^+) 
  \leq 8a \cdot d(z,\T)\cdot\dfrac{|z+1|}{|z-1|^3},
\end{equation}
and
\begin{equation}\label{cm22}
\dfrac{\sqrt{a}}{4}\cdot \dfrac{d(\la,\R^+)}{\sqrt{|\la|}(a+|\la|)} 
\leq d(z,\T)\leq 4\sqrt{a}\cdot \dfrac{d(\la,\R^+)}{\sqrt{|\la|}(a+|\la|)}.
\end{equation}
\end{proposition}

\begin{proof}
The first inequality is a direct application of Koebe distortion Theorem \cite[Corollary 1.4]{Po} to the map $\varphi_a$, so the proof is omitted.

For the second one, we have
\begin{equation}\label{cm}
|z+1| = \dfrac{2\sqrt{|\la|}}{|\sqrt{\la}+\mathrm{i}\sqrt{a}|},
\qquad 
|z-1| = \dfrac{2\sqrt{a}}{|\sqrt{\la}+\mathrm{i}\sqrt{a}|}.
\end{equation}
On the other hand,
$|\sqrt{\la}+\mathrm{i}\sqrt{a}|^2=|\la|+a+2\sqrt{a}\,\mathrm{Im}(\sqrt{\la})$,
and, since $\mathrm{Im}(\sqrt{\la}) \geq 0$, we obtain
\begin{equation*}
|\la|+a \leq |\sqrt{\la}+\mathrm{i}\sqrt{a}|^2\leq
 \left(|\sqrt{\la}|+|\sqrt{a}|\right)^2 \leq 2(a+|\la|).
\end{equation*}
Going back to inequalities \eqref{cm21}, we get \eqref{cm22}.
\end{proof}


\subsection{Schatten classes and determinants.}\label{s2.3}

One can find the definitions and properties of Schatten classes and regularized determinants related to these classes in \cite{DeHaKa} or \cite{Du}. 
For detailed discussion and proofs, see the monographs by Gohberg-Krein \cite{GoKr} and Simon \cite{Si1}.

Let us consider the following operator 
\begin{equation}\label{def:F}
F(\la):=(\la+a)(a + H)^{-1}V(\la-H_0)^{-1},
\end{equation}
where $a$ is large enough to guarantee that $(a+H)$ is invertible.
The coming Proposition \ref{det-BS} implies that
$V(\la-H_0)^{-1}\in \Sg_p$
for $\la \in \rho(H_0)$, provided $V\in L^p(\R^d)$ and $p>\max\{1;\frac{d}{2s}\}$.

For $\la \in \rho(H_0)$, we have $F(\la)\in \Sg_p$ and
$F$ holomorphic in $\rho(H_0)$, therefore
the holomorphic function of interest is, for all $\la \in \rho(H_0)$,
\begin{equation}\label{def:f}
f(\la):={\det}_{\lceil p \rceil}(\Id-F(\la)),
\end{equation}
where $\lceil p \rceil = \min\{n \in \N, n\geq p \}$.

In particular, the zeros of $f$ are the eigenvalues of $H$ (counted with algebraic multiplicities),
and, for $A \in \Sg_p$, we have
\begin{equation}\label{i-det}
|{\det}_{\lceil p \rceil}(\Id - A)| \leq \exp\left( \Gamma_p \|A\|^p_{\Sg_p}\right).
\end{equation}

We also use the well known inequality from \cite[Theorem 4.1]{Si2},
which we call Birman-Solomyak inequality; some authors prefer to call it Kato-Seiler-Simon inequality.
Observe that this inequality holds true for $1<p\leq 2$ by duality of the case $p\geq 2$.

\begin{proposition}\label{det-BS}
Let $V\in L^p(\R^d)$ complex-valued with $p > \max\{1;\frac{d}{2s}\}$, and $s>0$.
Assume that $\la\in\rho(H_0)$.

Then $V(\la-H_0)^{-1}\in \Sg_p$, and
\[
\|V(\la-H_0)^{-1}\|_{\Sg_p}^p  \leq
(2\pi)^{-d}\|V\|_{L^p}^p\cdot\|(\la-|\cdot |^{2s})^{-1}\|_{L^p}^p.
\]
\end{proposition}


\section{Bound on the resolvent.}\label{s-BS}

In this section, we bound the expression
$\|(\la-|\cdot |^{2s})^{-1}\|_{L^p}$ appearing in Proposition \ref{det-BS}.
The difficulty is to obtain the ``right'' bound when $s > \frac{d}{2}$.
Indeed, when $0<s \leq \frac{d}{2}$, we simply adapt the proof from \cite{DeHaKa} to dimensions $d\geq 1$.
The proof of the bound when $s> \frac{d}{2}$ requires more work.

We will repeatedly use the following elementary inequalities.
\begin{lemma}\label{i-lem}\hfill

\begin{enumerate}
\item Let $a,b \geq 0$, and $\al > 0 $, then
\begin{equation*}
\min\{1;2^{\al-1}\} (a^{\al}+b^{\al}) \leq (a+b)^{\al} \leq \max\{1;2^{\al-1}\} (a^{\al}+b^{\al}).
\end{equation*}

\item In particular, with $\al=2$, for $a,b \geq 0$, we have
\begin{equation*}
\sqrt{a^2+b^2}\leq a+b \leq \sqrt{2}\sqrt{a^2+b^2}.
\end{equation*}
\end{enumerate}
\end{lemma}

We recall that
$\mathsf{v}_{d-1}=\dfrac{2\pi^{\frac{d-1}{2}}}{\Gamma(\frac{d-1}{2})}$ for $d\geq 2$
and it is convenient to put $\mathsf{v}_{0}=2$ for $d=1$.

\begin{proposition}\label{det-BR}
Let  $\la=\la_0+ \mathrm{i} \la_1 \in \C\bsl\R^+$.
Set $\delta=\frac{d}{2s}-1$.

For $0<s\leq \frac{d}{2}$ \ and \ $p> \frac{d}{2s}$, we have
\begin{equation}\label{BR}
\left\|(\la-|\cdot |^{2s})^{-1} \right\|^p_{L^p} \leq \dfrac{\mathsf{v}_{d-1}}{2s} \cdot M_1 \cdot \dfrac{|\la|^{\frac{d}{2s}-1}}{d(\la,\si(H_0))^{p-1}},
\end{equation}
where
$M_1= \max\left\{K_2 ;
  \ds\int_0^{+\infty} \dfrac{t^{\delta}\,dt}{(t^2+1)^{\frac{p}{2}}} \right\}$,
and
$K_2$ is defined in \eqref{br-1} and depends on $d, p$, and $s$.

For $s>\frac{d}{2}$ \ and \ $p>1$, we have
\begin{equation}\label{BR1}
\left\|(\la-|\cdot |^{2s})^{-1} \right\|^p_{L^p} \leq \dfrac{\mathsf{v}_{d-1}}{2s} \cdot \dfrac{N_1}{d(\la,\si(H_0))^{p-\frac{d}{2s}}},
\end{equation}
where $N_1= \max\left\{
\ds\int_0^{+\infty}\dfrac{t^{\delta}}{(t^2+1)^{\frac{p}{2}}}\, dt
 ;\ds\int_0^1 t^{\delta}dt + 2\ds\int_0^{+\infty}
   \dfrac{dt}{(t^2+1)^{\frac{p}{2}}}\right\}$.
\end{proposition}

\begin{proof}
We start with the polar change of variables
\begin{equation*}
\left\|(\la-|x|^{2s})^{-1} \right\|^p_{L^p} = \mathsf{v}_{d-1} \ds\int_0^{+\infty} \dfrac{r^{d-1}}{|r^{2s}-\la|^p}\,dr,
\end{equation*}
and we put
\begin{equation}\label{e:int21}
I = \ds\int_0^{+\infty} \dfrac{r^{d-1}}{|r^{2s}-\la|^p}\,dr
  = \ds\int_0^{+\infty}
    \dfrac{r^{d-1}}{|(r^{2s}-\la_0)^2+\la_1^2|^{\frac{p}{2}}}\,dr.
\end{equation}


First, we assume  that $\la_0<0$,
that is, $d(\la,\si(H_0))=|\la|$.
In \eqref{e:int21}, we use
$(r^{2s}-\la_0)^2\geq r^{4s}+\la_0^2$, and we make the change of variables
$t=\dfrac{r^{2s}}{|\la|}$, so
\begin{align}
I \leq  \ds\int_{0}^{+\infty}
    \dfrac{r^{d-1}}{(r^{4s}+|\la|^2)^{\frac{p}{2}}}\,dr
   = \dfrac{1}{2s} \cdot \dfrac{|\la|^{\frac{d}{2s}}}{|\la|^p}
    \ds\int_{0}^{+\infty}
    \dfrac{t^{\frac{d}{2s}-1}}{(t^2+1)^{\frac{p}{2}}}\,dt. \label{cv:int}
\end{align}
The integrals in \eqref{cv:int} converge since $p>\frac{d}{2s}>0$.
Hence, for $\la_0 <0$,
\begin{equation}\label{br-0}
I \leq \dfrac{1}{2s} \cdot \ds\int_{0}^{+\infty}
    \dfrac{t^{\frac{d}{2s}-1}}{(t^2+1)^{\frac{p}{2}}}\,dt \cdot
    \dfrac{|\la|^{\frac{d}{2s}-1}}{d(\la,\si(H_0))^{p-1}}.
\end{equation}


Second, we assume $\la_0\geq 0$ and $\la_1>0$
(since $\|(\la-|x|^{2s})^{-1} \|=\|(\bar{\la}-|x|^{2s})^{-1}\|$).
In \eqref{e:int21},
we obtain with the help of  the change of variables  $t=\dfrac{r^{2s}-\la_0}{\la_1}$,
\begin{align}
I  & = \dfrac{1}{2s\la_1^{p-1}}\ds\int_{-\frac{\la_0}{\la_1}}^{+\infty}
\dfrac{(\la_1t+\la_0)^{\frac{d-1}{2s}}(\la_1t+\la_0)^{\frac{1}{2s}-1}}
{(t^2+1)^{\frac{p}{2}}}\,dt \nonumber \\
  & = \dfrac{1}{2s\la_1^{p-1}}\ds\int_{-\frac{\la_0}{\la_1}}^{+\infty}
\dfrac{(\la_1t+\la_0)^{\frac{d}{2s}-1}}
{(t^2+1)^{\frac{p}{2}}}\,dt. \label{e:int2}
\end{align} 
We distinguish two cases: $0< s \leq \frac{d}{2}$\  and \ $s> \frac{d}{2}$.

If $s=\frac{d}{2}$, for $\mathrm{Re}(\la)\geq 0$, bound \eqref{BR} is obvious from \eqref{e:int2}.


Now assume that $0<s<\frac{d}{2}$ and put $\delta=\frac{d}{2s}-1$.
Since $\la_1>0, -\frac{\la_0}{\la_1}<0$, we have
\begin{equation*}
\ds\int_{-\frac{\la_0}{\la_1}}^{+\infty}
\dfrac{(\la_1t+\la_0)^{\delta}}
{(t^2+1)^{\frac{p}{2}}}\,dt
 =  \ds\int_{-\frac{\la_0}{\la_1}}^{0} 
\dfrac{(\la_1t+\la_0)^{\delta}}
{(t^2+1)^{\frac{p}{2}}}\,dt 
+ 
\ds\int_{0}^{+\infty}
\dfrac{(\la_1t+\la_0)^{\delta}}
{(t^2+1)^{\frac{p}{2}}}\,dt.
\end{equation*}

In the first integral on the right-hand side of this equality, we use that
$\la_1t+\la_0 \leq \la_0$.
As for the second term, we observe that, by Lemma \ref{i-lem}, 
$(\la_1t+\la_0)^{\delta} \leq C_{d,s}
\big((\la_1t)^{\delta}+\la_0^{\delta}\big)$.
Here,  $C_{d,s}=\max\{1;2^{\delta-1}\}$.
Hence, we have
\begin{align*}
I & \leq \dfrac{1}{2s\la_1^{p-1}}
\left[\la_0^{\delta}\ds\int_{-\frac{\la_0}{\la_1}}^{0}
\dfrac{1}
{(t^2+1)^{\frac{p}{2}}}\,dt \;+\right. \\
&\hspace*{1.5cm} \left. +C_{d,s}\la_1^{\delta}
\ds\int_{0}^{+\infty}
\dfrac{t^{\delta}}
{(t^2+1)^{\frac{p}{2}}}\,dt
+C_{d,s}\la_0^{\delta}
\ds\int_{0}^{+\infty}
\dfrac{1}
{(t^2+1)^{\frac{p}{2}}}\,dt\right]\\
& \leq \dfrac{K_1}{2s\la_1^{p-1}}\cdot \left[\la_0^{\delta}+\la_1^{\delta} \right],
\end{align*}
where
$K_1=\max\left\{(1+C_{d,s})\ds\int_{\R^+}\dfrac{1}{(t^2+1)^{\frac{p}{2}}}\,dt;
C_{d,s}\ds\int_{\R^+}\dfrac{t^{\delta}}{(t^2+1)^{\frac{p}{2}}}\,dt\right\}$.
Then, putting
$C'_{d,s}=\max\{1;2^{1-\delta}\}$,
we have by Lemma \ref{i-lem},
\begin{align*}
I & \leq \dfrac{K_1}{2s\la_1^{p-1}}\cdot C'_{d,s} (\la_0+\la_1)^{\delta} \\
 & \leq \dfrac{K_1}{2s\la_1^{p-1}}\cdot C'_{d,s}(\sqrt{2})^{\delta} |\la|^{\delta}.
\end{align*}
Consequently, for $\la_0 \geq 0$,
\begin{equation}\label{br-1}
I \leq \dfrac{K_2}{2s} \cdot \dfrac{|\la|^{\frac{d}{2s}-1}}{d(\la,\si(H_0))^{p-1}},
\end{equation}
where
$K_2= K_1\cdot C'_{d,s}\cdot 2^{\delta/2}$.
Recalling \eqref{br-0}, we obtain inequality \eqref{BR} in the case $s\leq \frac{d}{2}$.


Let us turn now to the case  $s>\frac{d}{2}$.  Suppose again that $\la_1 >0$.
We recall \eqref{e:int2}
\begin{align*}
I & = \dfrac{1}{2s\la_1^{p-1}}\ds\int_{-\frac{\la_0}{\la_1}}^{+\infty}
\dfrac{(\la_1t+\la_0)^{\delta}}
{(t^2+1)^{\frac{p}{2}}}\,dt.
\end{align*}
Since $-1<\delta=\frac{d}{2s}-1<0$, we cannot use the previous bound.
Making the change of variables $u=t+\frac{\la_0}{\la_1}$, we obtain
\begin{align*}
\ds\int_{-\frac{\la_0}{\la_1}}^{+\infty}
\dfrac{(\la_1t+\la_0)^{\delta}}
{(t^2+1)^{\frac{p}{2}}}\,dt
  & = \la_1^{\delta} \ds\int_{-\frac{\la_0}{\la_1}}^{+\infty} 
\dfrac{(t+\frac{\la_0}{\la_1})^{\delta}}
{(t^2+1)^{\frac{p}{2}}}\,dt \\
& = \la_1^{\delta}
\ds\int_{0}^{+\infty}
\dfrac{u^{\delta}}
{\left((u-\frac{\la_0}{\la_1})^2+1\right)^{\frac{p}{2}}}\,du.
\end{align*}
The last integral is bounded independently of $\lambda$, {\it  i.e., }
\begin{align*}
\ds\int_{0}^{+\infty}
\dfrac{u^{\delta}}
{\left((u-\frac{\la_0}{\la_1})^2+1\right)^{\frac{p}{2}}}\,du
 & \leq \ds\int_0^1 u^{\delta}du
   + \ds\int_1^{+\infty}
   \dfrac{1}{\left((u-\frac{\la_0}{\la_1})^2+1\right)^{\frac{p}{2}}}\,du\\
 & \leq \ds\int_0^1 u^{\delta}du + \ds\int_{\R}
   \dfrac{1}{(u^2+1)^{\frac{p}{2}}}\,du.
\end{align*}
Indeed, in the first inequality, we use
$(u-\frac{\la_0}{\la_1})^2+1\geq 1$ when $0\leq u \leq 1$,
and $u^{\delta} \leq 1$ when $u\geq 1$ (since $\delta<0$ and $p>1$).
Hence, for $\la_0 \geq 0$,
\begin{equation}\label{br-2}
I \leq \dfrac{K_3}{2s}\cdot\dfrac{\la_1^{\delta}}{\la_1^{p-1}}=
\dfrac{K_3}{2s}\cdot\dfrac{1}{d(\la,\si(H_0))^{p-\frac{d}{2s}}},
\end{equation}
with $K_3=\ds\int_0^1 u^{\delta}du + 2\ds\int_0^{+\infty}
   \dfrac{du}{(u^2+1)^{\frac{p}{2}}}$.
Recalling \eqref{br-0}, the proof of \eqref{BR1} is finished.
\end{proof}


\section{Proof of Theorem \ref{T1}.}\label{s-proof1}

Reminding \eqref{def:f}, we have
$f(\la)={\det}_{\lceil p \rceil}(\Id-F(\la))$ for $\la \in \rho(H_0)=\C\bsl\R^{+}$,
where
\[F(\la):=(\la+a)(a+H)^{-1}V(\la-H_0)^{-1} \in \Sg_p,\quad  p\geq 1.\]
Inequality \eqref{i-det} implies that
\begin{equation*}
\log(|f(\la)|) \leq
\Gamma_p \|(\la+a)(a+H)^{-1}V(\la-H_0)^{-1}\|_{\Sg_p}^p. 
\end{equation*}
From \cite[Lemma 3.3.4]{Ha}, we know the following bound on $\|(-a-H)^{-1}\|$ for some $a>0$.
Then, there exists $\omega \geq 1$, depending on $d, p, s$, and $V$,
such that for any $a\geq \omega$
\begin{equation}\label{maj:opp}
\|(-a-H)^{-1}\| \leq \dfrac{C_{\omega}}{|\omega-a|},
\end{equation}
where $C_{\omega}=(1- \|V(-\omega-H_0)^{-1}\|)^{-1}$.
By Proposition \ref{det-BS}, we get the next inequality
for $p>1$ and $\la \in \C\bsl \R^{+}$
\begin{align}\label{i:proof}
\log |f(\la)| \leq \dfrac{\Gamma_p}{(2\pi)^d}\cdot
\dfrac{C_{\omega}^p}{|\omega-a|^p}\cdot |\la+a|^p \|V\|_{L^p}^p\|(\la-|\cdot |^{2s})^{-1} \|^p_{L^p}.
\end{align}


Since $0<s\leq \frac{d}{2}$, from \eqref{BR}, we have
\begin{equation}\label{aa1}
\log |f(\la)| \leq  \dfrac{K_1\, C_{\omega}^p}{|\omega-a|^p}  \cdot \|V\|_{L^p}^p\cdot 
\dfrac{|\la+a|^p|\la|^{\frac{d}{2s}-1}}{d(\la,\si(H_0))^{p-1}},
\end{equation}
where 
\begin{equation}\label{const:1}
K_1=\dfrac{\Gamma_p}{(2\pi)^d}\cdot \dfrac{\mathsf{v}_{d-1}}{2s}\cdot M_1
\end{equation}
and $M_1$ is defined in \eqref{BR}.


We now transfer the above inequality on $\D$ in order to apply Theorem \ref{BGK}.
That is, we consider the function
$g(z)=f\circ \varphi_a(z)$,
where $\varphi_a$ is defined by \eqref{phi}.
It is clearly holomorphic on $\D$.
By definition \eqref{phi},  we see
$|\la+a| = \dfrac{4a|z|}{|z-1|^2}$.
So, Proposition \ref{CM:2} applied to inequality \eqref{aa1} gives
\begin{align}
\log |g(z)| & \leq \dfrac{K_1\, C_{\omega}^p}{|\omega-a|^p} \cdot  \|V\|_{L^p}^p \cdot
 \dfrac{(4a)^p \,|z|^{p}\,a^{\frac{d}{2s}-1}|z+1|^{\frac{d}{s}-2}|z-1|^{3(p-1)}}
 {|z-1|^{2p}|z-1|^{\frac{d}{s}-2}a^{p-1}d(z,\T)^{p-1}|z+1|^{p-1}} \nonumber\\
  & \leq \dfrac{K_2\, a^{\frac{d}{2s}}}{|\omega-a|^p} \cdot \|V\|_{L^p}^p \cdot
 \dfrac{|z|^p}
 {d(z,\T)^{p-1}|z+1|^{p-\frac{d}{s}+1}|z-1|^{\frac{d}{s}-p+1}} \label{i-s1},
\end{align}
with $K_2=4^p K_1 C_{\omega}^p$.


Now, by Theorem \ref{BGK}, we have for all $\tau > 0$,
\begin{equation}\label{e:gen}
\ds\sum_{g(z)=0} (1-|z|)^{p+\tau} |z-1|^{(\frac{d}{s}-p+\tau)_+}
|z+1|^{(p-\frac{d}{s}+\tau)_+}
\leq C \cdot \dfrac{K_2\, a^{\frac{d}{2s}}}{|\omega-a|^p} \cdot \|V\|_{L^p}^p,
\end{equation}
where $K_2$ is defined above and $C$ depends on $d,p,s$, and $\tau$.
There are three cases to consider: Case 1: $\frac{d}{2s} < p < \frac{d}{s}$, 
Case 2: $p = \frac{d}{s}$,  and Case 3: $p > \frac{d}{s}$.

For each case we will transfer the relation \eqref{i-s1} back to $\rho(H_0)=\C\bsl\R^{+}$.
Recalling Proposition \ref{CM:2}, we see
\begin{align*}
  1-|z|=d(z,\T) & \geq
\dfrac{\sqrt{a}}{4} \cdot \dfrac{d(\la,\si(H_0))}{|\la|^{1/2}(a+|\la|)},\\
  |z+1|^2 \geq \dfrac{2|\la|}{a+|\la|},
& \quad \mbox{and} \quad
  |z-1|^2 \geq \dfrac{2a}{a+|\la|}.
\end{align*}
Then we will integrate the resulting inequality with respect to $a \in [\omega ; +\infty[$ to get to a sharper bound.
This integration follows the idea of Demuth, Hansmann, and Katriel
(see \cite{DeHaKa} or \cite{DeHaKa1}).

Since the computations are similar for all above cases, we give the details of the proof in the first case and present only the main steps in the remaining cases.


\subsection{Case 1: $\frac{d}{2s} < p < \frac{d}{s}$.}
For $0<\tau <\frac{d}{s}-p $, relation \eqref{e:gen} becomes
\begin{equation}\label{case1}
\ds\sum_{g(z)=0}(1-|z|)^{p+\tau}|z-1|^{\frac{d}{s}-p+\tau}
\leq C\cdot \dfrac{K_2\, a^{\frac{d}{2s}}}{|\omega-a|^p} \cdot \|V\|_{L^p}^p,
\end{equation}
We transfer this inequality back to $\rho(H_0)$, {\it i.e.,}
\[  (1-|z|)^{p+\tau}|z-1|^{\frac{d}{s}-p+\tau}
\geq 
  \dfrac{a^{\frac{d}{2s}+\tau}}{2^{\frac{5p}{2}+\frac{3\tau}{2}-\frac{d}{2s}}} \cdot 
  \dfrac{d(\la,\si(H_0))^{p+\tau}}
  {|\la|^{\frac{p+\tau}{2}}(a+|\la|)^{\frac{d}{2s}+\frac{p}{2}+\frac{3\tau}{2}}},
\]
so
\begin{equation}\label{i1:case1}
 \ds\sum_{\la \in \si_d(H)}  \dfrac{d(\la,\si(H_0))^{p+\tau}}
 {|\la|^{\frac{p+\tau}{2}}(a+|\la|)^{\frac{d}{2s}+\frac{p}{2}+\frac{3\tau}{2}}}
\lesssim  a^{-\tau}\cdot \dfrac{K_2}{|\omega-a|^p}\cdot
 2^{\frac{5p}{2}+\frac{3\tau}{2}-\frac{d}{2s}}\cdot \|V\|_{L^p}^p.
\end{equation}


We can now perform the integration with respect to $a \in [\omega; +\infty[$.
All terms in this relation are positive, so we can permute the sum and the integral by the Fubini Theorem.
Consequently, we obtain 
\begin{equation*}
\ds\sum_{\la \in \si_d(H)} 
 \dfrac{d(\la,\si(H_0))^{p+\tau}}{|\la|^{\frac{p+\tau}{2}}}
 \ds\int_{\omega}^{+\infty} \dfrac{|\omega-a|^p a^{-1}}{(a+|\la|)^{\frac{d}{2s}+\frac{p}{2}+\frac{3\tau}{2}}}\,da
\lesssim  \ds\int_{\omega}^{+\infty} \dfrac{da}{a^{1+\tau}} \cdot \|V\|_{L^p}^p.
\end{equation*}
Obviously, $\ds\int_{\omega}^{+\infty} \dfrac{da}{a^{1+\tau}}= \dfrac{1}{\tau \omega^{\tau}}$.
In the left-hand side of this relation, we use the bound $a \leq a +|\la|$
and we make the change of variables
$t =\dfrac{a-\omega}{|\la|+\omega}$. Hence, we come to
\begin{align*}
 \ds\int_{\omega}^{+\infty} \dfrac{|\omega-a|^p a^{-1}}{(a+|\la|)^{\frac{d}{2s}+\frac{p}{2}+\frac{3\tau}{2}}}\,da & \geq
  \ds\int_{\omega}^{+\infty} \dfrac{|\omega-a|^p }{(a+|\la|)^{\frac{d}{2s}+\frac{p}{2}+1+\frac{3\tau}{2}}}\,da \\
  & \geq \dfrac{(|\la|+\omega)^{p+1}}{(|\la|+\omega)^{\frac{d}{2s}+\frac{p}{2}+1+\frac{3\tau}{2}}}
   \ds\int_{0}^{+\infty} \dfrac{t^p \,dt}{(t+1)^{\frac{d}{2s}+\frac{p}{2}+1+\frac{3\tau}{2}}}.
\end{align*}
So, for $\frac{d}{2s}<p<\frac{d}{s}$,
\begin{equation}\label{i2:case1}
\ds\sum_{\la \in \si_d(H)}  \dfrac{d(\la,\si(H_0))^{p+\tau}}
 {|\la|^{\frac{p+\tau}{2}}(\omega+|\la|)^{\frac{d}{2s}-\frac{p}{2}+\frac{3\tau}{2}}}
\leq C \cdot \dfrac{K_1\, C_{\omega}^p}{I_1\, \tau \omega^{\tau}} \cdot 2^{\delta_1} \cdot \|V\|_{L^p}^p,
\end{equation}
where $K_1$ is defined in \eqref{const:1}, $C$ is defined in \eqref{e:gen},
$I_1= \ds\int_{0}^{+\infty} 
\dfrac{t^p \,dt}{(t+1)^{\frac{d}{2s}+\frac{p}{2}+1+\frac{3\tau}{2}}}$,
and $\delta_1=\frac{9p}{2}+\frac{3\tau}{2}-\frac{d}{2s}$.

\subsection{Case 2: $p=\frac{d}{s}$.}

Reminding  \eqref{e:gen}, we obtain 
\begin{equation}\label{case2}
\ds\sum_{g(z)=0}(1-|z|)^{p+\tau}|z-1|^{\tau}|z+1|^{\tau}
\leq C\cdot \dfrac{K_2\, a^{\frac{d}{2s}}}{|\omega-a|^p} \cdot \|V\|_{L^p}^p.
\end{equation}
Furthermore, we have
\[ (1-|z|)^{p+\tau}|z-1|^{\tau}|z+1|^{\tau}
\geq 
  \dfrac{a^{\frac{p}{2}+\tau}}{2^{2p+\tau}}\cdot
  \dfrac{d(\la,\si_d(H_0))^{p+\tau}}
  {|\la|^{\frac{p}{2}}(a+|\la|)^{p+2\tau}},\]
so, for all $0<\tau <1$,
\begin{equation}\label{i1:case2}
 \ds\sum_{\la \in \si(H)} \dfrac{d(\la,\si(H_0))^{p+\tau}}
 {|\la|^{\frac{p}{2}}(a+|\la|)^{p+2\tau}}
\lesssim a^{-\tau}\cdot \dfrac{K_2}{|\omega-a|^p}
\cdot 2^{2p+\tau} \cdot \|V\|_{L^p}^p.
\end{equation}
After integration, we find  
\begin{equation}\label{i2:case2}
\ds\sum_{\la \in \si_d(H)}  \dfrac{d(\la,\si(H_0))^{p+\tau}}
 {|\la|^{\frac{p}{2}}(\omega+|\la|)^{2\tau}}
\leq C \cdot \dfrac{K_1\, C_{\omega}^p}{I_2\, \tau \omega^{\tau}} \cdot 2^{\delta_2}\cdot
 \|V\|_{L^p}^p,
\end{equation}
where $K_1$ is defined in \eqref{const:1}, $C$ is defined in \eqref{e:gen},
$I_2=\ds\int_{0}^{+\infty} \dfrac{t^p}{(t+1)^{p+1+2\tau}} \,dt$,
and $\delta_2=4p+\frac{3}{2}\tau$.
In \eqref{i2:case2}, we used $1 \leq 2^{\tau/2}$, to obtain a clear statement of $\delta_j$ in Remark \ref{rem:const1}.

\subsection{Case 3: $p>\frac{d}{s}$.}

For $0<\tau< p-\frac{d}{s}$, relation \eqref{e:gen} becomes
\begin{equation}\label{case3}
\ds\sum_{g(z)=0}(1-|z|)^{p+\tau}|z+1|^{p-\frac{d}{s}+\tau} 
\leq C\cdot \dfrac{K_2\, a^{\frac{d}{2s}}}{|\omega-a|^p} \cdot \|V\|_{L^p}^p.
\end{equation}
Then, we get
\[ (1-|z|)^{p+\tau}|z+1|^{p-\frac{d}{s}+\tau}
\geq 
   \dfrac{a^{\frac{p+\tau}{2}}}{2^{\frac{3}{2}(p+\tau)+\frac{d}{2s}}}\cdot
   \dfrac{d(\la,\si(H_0))^{p+\tau}}
  {|\la|^{\frac{d}{2s}}(a+|\la|)^{\frac{3p}{2}-\frac{d}{2s}+\frac{3\tau}{2}}},\]
and, for $0<\tau< p-\frac{d}{s}$,
\begin{equation}\label{i1:case3}
 \ds\sum_{\la \in \si_d(H)} \dfrac{d(\la,\si(H_0))^{p+\tau}}
  {|\la|^{\frac{d}{2s}}(a+|\la|)^{\frac{3}{2}(p+\tau)-\frac{d}{2s}}}
\lesssim  a^{\frac{d}{2s}-\frac{p+\tau}{2}}\cdot \dfrac{K_2}{|\omega-a|^p}
\cdot 2^{\frac{3}{2}(p+\tau)+\frac{d}{2s}} \cdot \|V\|_{L^p}^p.
\end{equation}
We integrate the above inequality
\begin{equation*}
\ds\sum_{\la \in \si_d(H)}
 \dfrac{d(\la,\si(H_0))^{p+\tau}}{|\la|^{\frac{d}{2s}}}
 \ds\int_{\omega}^{+\infty} \dfrac{|\omega-a|^p a^{\frac{p+\tau}{2}-\frac{d}{2s}-1-\tau}}{(a+|\la|)^{\frac{3}{2}(p+\tau)-\frac{d}{2s}}}\,da
\lesssim  \dfrac{\|V\|_{L^p}^p}{\tau\omega^{\tau}}.
\end{equation*}
As before, we do the change of variables $t =\dfrac{a-\omega}{|\la|+\omega}$,
and we distinguish two cases:
if $\frac{1}{2}(p-\frac{d}{s}-2-\tau) <0$,
we  use the bound $[(|\la|+\omega)t+\omega] \leq (|\la|+\omega)(t+1)$,
and if $\frac{1}{2}(p-\frac{d}{s}-2-\tau) \geq 0$,
we apply the bound from below $[(|\la|+\omega)t+\omega] \geq (|\la|+\omega)t$.
We present the case $\frac{1}{2}(p-\frac{d}{s}-2-\tau) \geq 0$ in details,
the other cases are analogous and are omitted. We see
\begin{align*}
& \; \ds\int_{\omega}^{+\infty} \dfrac{|\omega-a|^p a^{\frac{p}{2}-\frac{d}{2s}-1-\frac{\tau}{2}}}{(a+|\la|)^{\frac{3}{2}(p+\tau)-\frac{d}{2s}}}\,da  = \\
= & \; \dfrac{(|\la|+\omega)^{p+1}}{(|\la|+\omega)^{\frac{3}{2}(p+\tau)-\frac{d}{2s}}}
   \ds\int_{0}^{+\infty} \dfrac{t^p[(|\la|+\omega)t+\omega]^{\frac{p}{2}-\frac{d}{2s}-1-\frac{\tau}{2}}}{(t+1)^{\frac{3}{2}(p+\tau)-\frac{d}{2s}}}\,dt \\
\geq   & \; \dfrac{(|\la|+\omega)^{\frac{p}{2}-\frac{d}{2s}-1-\frac{\tau}{2}}}
   {(|\la|+\omega)^{\frac{p}{2}-\frac{d}{2s}-1+\frac{3\tau}{2}}}
   \ds\int_{0}^{+\infty} \dfrac{t^{\frac{3p}{2}-\frac{d}{2s}-1-\frac{\tau}{2}}}{(t+1)^{\frac{3}{2}(p+\tau)-\frac{d}{2s}}}\,dt.
\end{align*}
So, if $\frac{1}{2}(p-\frac{d}{s}-2-\tau)\geq 0$, we obtain
\begin{equation}\label{i2:case3a}
\ds\sum_{\la \in \si_d(H)}  \dfrac{d(\la,\si(H_0))^{p+\tau}}
 {|\la|^{\frac{d}{2s}}(\omega+|\la|)^{2\tau}}
\leq C \cdot \dfrac{K_1\, C_{\omega}^p}{I_3\, \tau \omega^{\tau}}  \cdot 2^{\delta_3} \cdot \|V\|_{L^p}^p,
\end{equation}
where $K_1$ is defined in \eqref{const:1}, $C$ in \eqref{e:gen},
$I_3= \ds\int_{0}^{+\infty} \dfrac{t^{\frac{3p}{2}-\frac{d}{2s}-1-\frac{\tau}{2}}}{(t+1)^{\frac{3}{2}(p+\tau)-\frac{d}{2s}}}\,dt$,
and $\delta_3=\frac{7}{2}p+\frac{d}{2s}+\frac{3}{2}\tau$.

When $\frac{1}{2}(p-\frac{d}{s}-2-\tau) <0$, we obtain
\begin{equation}\label{i2:case3b}
\ds\sum_{\la \in \si_d(H)}  \dfrac{d(\la,\si(H_0))^{p+\tau}}
 {|\la|^{\frac{d}{2s}}(\omega+|\la|)^{2\tau}}
\leq C \cdot \dfrac{K_1\, C_{\omega}^p}{I_4\, \tau \omega^{\tau}}  \cdot 2^{\delta_4} \cdot \|V\|_{L^p}^p,
\end{equation}
where $K_1$ is defined in \eqref{const:1}, $C$ is defined in \eqref{e:gen},
$I_4= \ds\int_{0}^{+\infty} \dfrac{t^{p}}{(t+1)^{p+1+2\tau}}\,dt$,
and $\delta_4=\delta_3=\frac{7}{2}p+\frac{d}{2s}+\frac{3}{2}\tau$.

In relations \eqref{i2:case1}, \eqref{i2:case2}, \eqref{i2:case3a}, and \eqref{i2:case3b},
we use the bound $\omega +|\la| \leq \omega(1+|\la|)$, because $\omega \geq 1$,
and we come to inequality \eqref{t11}.
Thus, the proof in the case $0<s\leq \frac{d}{2}$ is finished.
\hfill $\Box$

\begin{remark}\label{rem:const1}
In the above inequalities, one has
\begin{align*}
I_j & = \ds\int_{0}^{+\infty} \dfrac{t^{p+\frac{1}{2}(p-\frac{d}{s}-2-\tau)_+}}
{(t+1)^{p+1+2\tau+\frac{1}{2}\max\{\frac{d}{s}-p-2\tau; 0; p-\frac{d}{s}-2-\tau\}}} \,dt,\\
\delta_j&  =\frac{7p}{2}+\frac{3\tau}{2}+\min\{p; \frac{d}{s}\}-\frac{d}{2s},
\end{align*}
where $j=1,\dots, 4$.
\end{remark}


\section{Proof of Theorem \ref{T1b}.}\label{s-proof1b}

Taking into account \eqref{BR1} and \eqref{maj:opp},
inequality \eqref{i:proof} becomes for $\la\in \rho(H_0)$
\begin{equation*}
\log |f(\la)| \leq \dfrac{K_4\, C_{\omega}^p}{|\omega-a|^p} \cdot \|V\|_{L^p}^p \cdot
\dfrac{|\la+a|^p}{d(\la,\si(H_0))^{p-\frac{d}{2s}}},
\end{equation*}
since $s>\frac{d}{2}$ and we have
\begin{equation}\label{const:2}
K_4=\dfrac{\Gamma_p}{(2\pi)^d}\cdot  \dfrac{\mathsf{v}_{d-1}}{2s}\cdot N_1,
\end{equation}
where $N_1$ depends on $d, p$, and $s$ only.
As before, we have
\begin{equation*}
\log|g(z)|  \leq \dfrac{K_4\, C_{\omega}^p}{|\omega-a|^p}\cdot
\dfrac{4^p a^{\frac{d}{2s}}|z|^p}{d(z,\T)^{p-\frac{d}{2s}} |z-1|^{\frac{3d}{2s}-p}
  |z+1|^{p-\frac{d}{2s}}}.
\end{equation*}
We set $K_5=4^p K_4\, C_{\omega}^p$. 
Applying Theorem \ref{BGK}, we have
\begin{align}\label{e:gen2}
\ds\sum_{g(z)=0} (1-|z|)^{p-\frac{d}{2s}+1+\tau}
 |z-1|^{(\frac{3d}{2s}-p-1+\tau)_+}
 & |z+1|^{(p-\frac{d}{2s}-1+\tau)_+} \leq \\
& \leq \dfrac{C \, K_5\, a^{\frac{d}{2s}}}{|\omega-a|^p} \cdot \|V\|_{L^p}^p. \nonumber
\end{align}
The separation in different cases with respect to $p$ and $\frac{d}{2s}$ is clear from the following picture (Figure \ref{fig}).
The x-axis represents $p$ and the y-axis represents $\frac{d}{2s}$.
There are four straight lines  given by $y=1$, $x-y-1=0$, $-x+3y-1=0$,
and $x-3y-1=0$.


\begin{figure}[htbp]
\includegraphics[width=10cm]{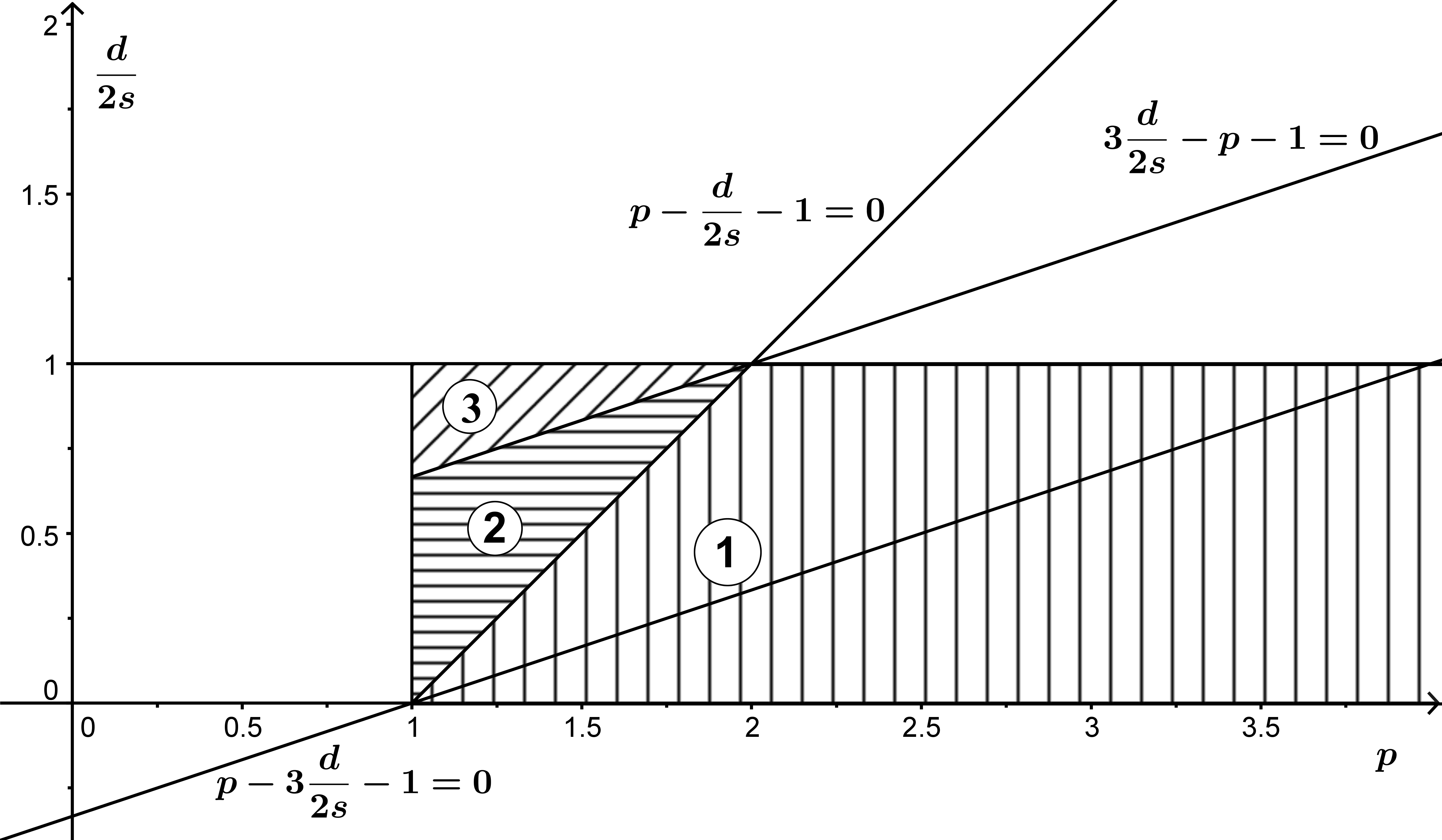}
\caption{The different cases}
\label{fig}
\end{figure}

So, we have three different cases to consider: Case 1: 
$p-\frac{d}{2s}-1 \geq 0$ \ and \ $\frac{3d}{2s}-p-1 < 0$, Case 2:  $p-\frac{d}{2s}-1 < 0$ and $\frac{3d}{2s}-p-1 < 0$, and
Case 3: $p-\frac{d}{2s}-1 < 0$ and $\frac{3d}{2s}-p-1\geq 0$.
Below, the computations are similar to the case $s\leq \frac{d}{2}$, so they are omitted.


\subsection{Case 1: $p-\frac{d}{2s}-1 \geq 0$ and $\frac{3d}{2s}-p-1 < 0$.}
We find
\begin{equation}\label{i1:case4}
\ds\sum_{\la \in \si_d(H)}
  \dfrac{d(\la,\si(H_0))^{p-\frac{d}{2s}+1+\tau}}
  {|\la|(a+|\la|)^{3\frac{p+\tau}{2}-\frac{3d}{4s}+\frac{1}{2}}}
\lesssim \dfrac{a^{-\frac{1}{2}(p+1-\frac{3d}{2s}+\tau)}}{|\omega-a|^p}
 \cdot \|V\|_{L^p}^p.
\end{equation}
It remains to integrate with respect to $a$ on $[\omega;+\infty[$.
We do it in the same way as in the case $0<s\leq \frac{d}{2}$. 

From \eqref{i1:case4}, we obtain
\begin{equation*}
\ds\sum_{\la \in \si_d(H)} \dfrac{d(\la,\si(H_0))^{p+1-\frac{d}{2s}+\tau}}{|\la|}
 \ds\int_{\omega}^{+\infty} \dfrac{|\omega-a|^p a^{\frac{1}{2}(p-\frac{3d}{2s}-1-\tau)}}{(a+|\la|)^{3\frac{p+\tau}{2}-\frac{3d}{4s}+\frac{1}{2}}}\,da
\lesssim  \dfrac{\|V\|_{L^p}^p}{\tau\omega^{\tau}}.
\end{equation*}
Hence, if $p-\frac{3d}{2s}-1 > 0$, and  $\tau>0$ is small enough,
\begin{equation}\label{i2:case4a}
\ds\sum_{\la \in \si_d(H)}  \dfrac{d(\la,\si(H_0))^{p+1-\frac{d}{2s}+\tau}}
 {|\la|(\omega+|\la|)^{2\tau}}
\leq C\cdot \dfrac{K_4\, C_{\omega}^p}{I_5\, \tau \omega^{\tau}}  \cdot 2^{\delta_5} \cdot \|V\|_{L^p}^p,
\end{equation}
where $K_4$ is defined in \eqref{const:2}, $C$ in \eqref{e:gen2},
$I_5=\ds\int_{0}^{+\infty} \dfrac{t^{\frac{1}{2}(3p-\frac{3d}{2s}-1-\tau)}} {(t+1)^{3\frac{p+\tau}{2}-\frac{3d}{4s}+\frac{1}{2}}}\,dt$,
and $\delta_5=\frac{1}{2}(7p+5-\frac{3d}{2s}+3\tau)$.

Otherwise, if $p-\frac{3d}{2s}-1 \leq 0$, we have
\begin{equation}\label{i2:case4b}
\ds\sum_{\la \in \si_d(H)}  \dfrac{d(\la,\si(H_0))^{p+1-\frac{d}{2s}+\tau}}
 {|\la|(\omega+|\la|)^{2\tau}}
\leq C\cdot \dfrac{K_4\, C_{\omega}^p}{I_6\, \tau \omega^{\tau}}  \cdot 2^{\delta_6} \cdot \|V\|_{L^p}^p,
\end{equation}
where $K_4$ is defined in \eqref{const:2}, $C$ is defined in \eqref{e:gen2},
$I_6= \ds\int_{0}^{+\infty} \dfrac{t^p} {(t+1)^{p+1+2\tau}}\,dt $,
and $\delta_6=\delta_5=\frac{1}{2}(7p+5-\frac{3d}{2s}+3\tau)$.


\subsection{Case 2: $p-\frac{d}{2s}-1 < 0$ and $\frac{3d}{2s}-p-1 < 0$.}
We have
\begin{equation}\label{i1:case5}
\ds\sum_{\la \in \si_d(H)}\dfrac{d(\la,\si(H_0))^{p-\frac{d}{2s}+1+\tau}}
  {|\la|^{\frac{1}{2}(p+1-\frac{d}{2s}+\tau)}(a+|\la|)^{p+1-\frac{d}{2s}+\tau}}
\lesssim \dfrac{a^{-\frac{1}{2}(p+1-\frac{3d}{2s}+\tau)}}{|\omega-a|^p} 
 \cdot \|V\|_{L^p}^p.
\end{equation}
Integrating this inequality gives
\begin{equation}\label{i2:case5}
\ds\sum_{\la \in \si_d(H)}  \dfrac{d(\la,\si(H_0))^{p+1-\frac{d}{2s}+\tau}}
 {|\la|^{\frac{1}{2}(p+1-\frac{d}{2s}+\tau)}
 (\omega+|\la|)^{\frac{1}{2}(\frac{d}{2s}-p+1+3\tau)}}
\leq C\cdot \dfrac{K_4\, C_{\omega}^p}{I_7\, \tau \omega^{\tau}} \cdot 2^{\delta_7} \cdot \|V\|_{L^p}^p,
\end{equation}
where $K_4$ is defined in \eqref{const:2}, $C$ in \eqref{e:gen2},
$I_7=\ds\int_{0}^{+\infty} \dfrac{t^p\, dt}{(t+1)^{\frac{1}{2}(p+\frac{d}{2s}+3+3\tau)}}$, and
$\delta_7=2(2p+1-\frac{d}{2s}+\tau)$.
We recall that $0< p-\frac{d}{2s} < 1$, hence $\frac{d}{2s}-p+1>0$.

\subsection{Case 3: $p-\frac{d}{2s}-1 < 0$ and $\frac{3d}{2s}-p-1\geq 0$.}
This time, we have
\begin{equation}\label{i1:case6}
\ds\sum_{\la \in \si_d(H)}
  \dfrac{d(\la,\si(H_0))^{p-\frac{d}{2s}+1+\tau}}
  {|\la|^{\frac{p+1+\tau}{2}-\frac{d}{4s}}
  (a+|\la|)^{\frac{p+1}{2}+\frac{d}{4s}+\frac{3\tau}{2}}}
\lesssim  \dfrac{a^{-\tau} }{|\omega-a|^p} \cdot \|V\|_{L^p}^p.
\end{equation}
After integration, the previous inequality becomes
\begin{equation}\label{i2:case6}
\ds\sum_{\la \in \si_d(H)}  \dfrac{d(\la,\si(H_0))^{p+1-\frac{d}{2s}+\tau}}
 {|\la|^{\frac{1}{2}(p+1-\frac{d}{2s}+\tau)}(\omega+|\la|)^{\frac{1}{2}(\frac{d}{2s}-p+1+3\tau)}}
\leq C\cdot \dfrac{K_4\, C_{\omega}^p}{I_8\, \tau \omega^{\tau}} \cdot
2^{\delta_8} \cdot \|V\|_{L^p}^p,
\end{equation}
where $K_4$ is defined in \eqref{const:2}, $C$ in \eqref{e:gen2},
$I_8=\ds\int_{0}^{+\infty} \dfrac{t^p \,dt}
{(t+1)^{\frac{1}{2}(p+3+\frac{d}{2s}+3\tau)}}$,
and $\delta_8= \frac{9}{2}p+\frac{5}{2}-\frac{7d}{4s}+\frac{3}{2}\tau$.
As before,  $0< p-\frac{d}{2s} < 1$, and so $\frac{d}{2s}-p+1>0$.

To make the statement of the theorem more transparent in Case 2,
we use $\dfrac{1}{(1+|\la|)^{\tau}} \geq \dfrac{1}{(1+|\la|)^{\frac{3\tau}{2}}}$.
This gives the power $\beta$ in relation \eqref{t12}.
Finally, since $\omega \geq 1$, we bound $\omega +|\la| \leq w(1+|\la|)$.
The proof of Theorem \ref{T1b} is finished.
\hfill $\Box$

\begin{remark}
In the above inequalities, one has
\begin{align*}
I_j & = \ds\int_0^{+\infty} \dfrac{t^{p+\frac{1}{2}(p-\frac{3d}{2s}-1-\tau)_+}}
{(t+1)^{p+1+2\tau+\frac{1}{2}\max\{p-\frac{3d}{2s}-1-\tau
; 0; \frac{d}{2s}+1-p-\tau\}}} \,dt, \\
\delta_j & = 2(2p+1-\frac{d}{2s}+\tau)-\frac{1}{2}\max\{p-\frac{d}{2s}-1+\tau;0;\frac{3d}{2s}-p-1+\tau\},
\end{align*}
where $j=5,\dots, 8.$
\end{remark}


\section{Lieb-Thirring bound using a theorem from \cite{Ha1}.}\label{s-proofH}

\subsection{Hansmann's Theorem and conformal mapping.}

The following theorem is the key ingredient for the proof of \eqref{t2}.
It is proved in \cite{Ha1}.

\begin{theorem}\label{TH}
Let $A$ be a normal bounded operator and $B$ be an operator such that
$B-A \in \Sg_p$ for some $p\geq 1$.
Suppose also that $\si(A)$ is convex. Then the following inequality holds
\begin{equation*}
\ds\sum_{\la \in \si_d(B)} d\left(\la,\si(A)\right)^p
\leq \| B-A \|_{\Sg_p}^p.
\end{equation*}
\end{theorem}

Hansmann has  another result in this direction (see \cite[Cor. 1]{Ha2}).
Since in our situation  the spectrum $\si(A)$ is convex, it gives no improvement.


We will apply Theorem \ref{TH} to $(-a-H)^{-1}$ and $(-a-H_0)^{-1}$, and so we introduce the parameter $a>0$.
As in previous section, we need a distortion result.
Introduce a conformal map
$g:\C\backslash \R^+ \rightarrow \bar\C\backslash [-\frac{1}{a},0]$
defined by
\[g(\la)=\dfrac{-1}{a+\la}.\]
Below,  we denote by $\la$ and $\mu$ the variables in
$\C\bsl \R^+$ and $\bar\C\bsl [-\frac{1}{a},0]$, respectively.

\begin{proposition}\label{dist}
For $\la \in \C\bsl \R^+$, we have the following bound
\[ d\left(g(\la),\left[-\dfrac{1}{a},0 \right]\,\right) \geq
\dfrac{1}{2\sqrt{5}}\cdot\dfrac{d\left(\la,\R^+\right)}{(a+|\la|)^2}.\]
\end{proposition}

\begin{proof}
We obtain a bound for the function
$\widetilde{g}: \C\backslash [a;+\infty[ \rightarrow \bar\C\backslash [0;\frac{1}{a}]$ defined by $\widetilde{g}(\la)=\frac{1}{\la}$ and then compose it by the translation $T:\la\mapsto \la+a$, that is $g=-\widetilde{g}\circ T$.
After some technical computations (see \cite{Du3}), we  find
\begin{equation*}
d\left(\widetilde{g}(\la),\left[0;\frac{1}{a} \right]\,\right) \geq
\dfrac{1}{\sqrt{5}}\cdot\dfrac{d\left(\la,[a;+\infty[\right)}{|\la|\cdot (a+|\la|)}.
\end{equation*}
The claimed inequality follows.
\end{proof}


As in proofs of Theorems \ref{T1} and \ref{T1b}, we use an integration with respect to the parameter $a$,
to improve the rate of convergence in the left-hand side of inequality \eqref{t2}.
This  trick is borrowed from  Theorem 5.3.3 in \cite{DeHaKa1}.
We recall that $\omega$ is defined in \eqref{maj:opp}.

\subsection{Proof of Theorem \ref{T2}.}
We put $A=(-a-H_0)^{-1}$ which is normal and $B=(-a-H)^{-1}$ which is bounded,
for $a > \omega $, so that $A$ and $B$ exist.
We know that $B-A=BVA \in \Sg_p$,
hence we can apply Theorem \ref{TH}.
For $p\geq 1$, it gives
\begin{equation}\label{eq1}
\ds\sum_{\mu \in \si_d(B)} d\left(\mu,\si(A)\right)^p
\leq \| B-A \|_{\Sg_p}^p.
\end{equation}
For $p>\max\{1;\frac{d}{2s}\}$, we bound the right-hand side of inequality \eqref{eq1}
with the help of Proposition \ref{det-BS}, and the inequalities \eqref{maj:opp} and \eqref{br-0}
\begin{align*}
\| B-A \|_{\Sg_p}^p & \leq (2\pi)^{-d}
\|(-a-H)^{-1}\|^p \cdot \|V\|_{L^p}^p \cdot \|(-a-|x|^{2s})^{-1}\|_{L^p}^p \nonumber\\
  & \leq  K_1 \, C_{\omega}^p \cdot \dfrac{a^{\frac{d}{2s}-p} }{|\omega-a|^p} 
  \cdot \|V\|_{L^p}^p,
\end{align*}
where
\begin{equation}\label{const:3}
K_1= \dfrac{\mathsf{v}_{d-1}}{2s(2\pi)^d}  \cdot \ds\int_{\R^+} \dfrac{t^{\frac{d}{2s}-1}}{(t^2+1)^{p/2}}\,dt.
\end{equation}
Then
$\mu=(-a-\la)^{-1}=g(\la) \in \si_d(B)$ if and only if
$\la \in \si_d(H)$, hence
\begin{align*}
\ds\sum_{\mu \in \si_d(B)} d\left(\mu,\si(A)\right)^p & =
\ds\sum_{\{g(\la), \la \in \si_d(H)\}} d\left(g(\la),\si(A)\right)^p \nonumber\\
& \geq \dfrac{1}{(2\sqrt{5})^{p}} \cdot \ds\sum_{\la \in \si_d(H)}
\dfrac{d\left(\la,\si(H_0)\right)^p}{(a+|\la|)^{2p}}.
\end{align*}
The last inequality results from Proposition \ref{dist}.
Thus, we obtain
\begin{equation}\label{i:Lf}
 \ds\sum_{\la \in \si_d(H)} \dfrac{d(\la,\si(H_0))^{p}}
{(a+|\la|)^{2p}}
\leq  (2\sqrt{5})^p\, K_1 \cdot C_{\omega}^p \cdot \dfrac{a^{\frac{d}{2s}-p}}{|\omega-a|^p}  \cdot \|V\|^p_{L^p},
\end{equation}
where $K_1$ is defined in \eqref{const:3}.


The next step of the proof is the integration with respect to parameter $a$.
Since the computations are similar to the integration performed in Section \ref{s-proof1},
the technical details are omitted.
We obtain from  \eqref{i:Lf}
\begin{equation*}
\ds\sum_{\la \in \si_d(H)} 
d(\la,\si(H_0))^{p} \ds\int_{\omega}^{+\infty} 
\dfrac{a^{p-\frac{d}{2s}-1-\tau}|\omega-a|^p}{(a+|\la|)^{2p}} \,da
   \lesssim  \dfrac{\|V\|^p_{L^p}}{\tau \omega^{\tau}}.
\end{equation*}

Hence, assuming first that $p-d/2s > 1$, we come to
\begin{equation*}
\ds\int_{\omega}^{+\infty} 
\dfrac{a^{p-\frac{d}{2s}-1-\tau}|a-\omega|^p da}{(a+|\la|)^{2p}}
\geq \dfrac{1}{(|\la|+\omega)^{\frac{d}{2s}+\tau}}
\ds\int_{0}^{+\infty} \dfrac{t^{2p-\frac{d}{2s}-1-\tau}}{(t+1)^{2p}} \,dt.
\end{equation*}
When $p-d/2s <1$, we have
\begin{equation*}
\ds\int_{\omega}^{+\infty} 
\dfrac{a^{p-\frac{d}{2s}-1-\tau}|a-\omega|^p da}{(a+|\la|)^{2p}} 
 \geq \dfrac{1}{(|\la|+\omega)^{\frac{d}{2s}+\tau}}
\ds\int_{0}^{+\infty} \dfrac{t^p \,dt}{(t+1)^{p+\frac{d}{2s}+1+\tau}}.
\end{equation*}
Hence
\begin{equation*}
\ds\sum_{\la \in \si_d(H)} 
\dfrac{d(\la,\si(H_0))^{p}}{(\omega+|\la|)^{\frac{d}{2s}+\tau}}
   \leq (2\sqrt{5})^p \cdot \dfrac{K_1\, C_{\omega}^p}{I \, \tau \omega^{\tau}} \cdot \|V\|^p_{L^p},
\end{equation*}
where $K_1$ is defined in \eqref{const:3}, and
\[I=\ds\int_{0}^{+\infty} \dfrac{t^{p+(p-\frac{d}{2s}-1-\tau)_+}}{(t+1)^{p+\frac{d}{2s}+1+\tau+(p-\frac{d}{2s}-1-\tau)_+}} \,dt.\]
Using $\omega +|\la| \leq \omega (1+|\la|)$,
the proof of Theorem \ref{T2} is complete.
\hfill $\Box$


\bibliographystyle{alpha}
\bibliography{biblio}

\end{document}